\documentclass[a4paper]{scrartcl}
\usepackage{graphicx}
\usepackage{color}
\usepackage{url}
\usepackage{parskip}
\usepackage{booktabs}
\usepackage{amsmath,amssymb,amsfonts,amssymb}
\usepackage{amsthm}
\usepackage{subcaption}
\usepackage{multirow}

\usepackage[margin=1.0in]{geometry}
\usepackage[shortlabels]{enumitem}  

\usepackage{algorithm}
\usepackage[noend]{algpseudocode}

\title{Computationally Efficient Dynamic Traffic Optimization Of Railway Systems}
\author{
	Robin~Vujanic,
	and Andrew~J.~Hill,
	\thanks{R. Vujanic and A. Hill are with the Australian Centre for Fields Robotics,
		School of Aerospace, Mechanical and Mechatronic Engineering, The University of Sydney, Australia, e-mail: robin@acfr.usyd.edu.au}
	\thanks{Manuscript received November 27, 2019; revised --.}
}

\newcommand{\tim}{t}
\newcommand{\stag}{k}

\newcommand{\Pp}{\mathcal{P}}
\newcommand{\Pfitzr}{\mathcal{P}(\tim, \xt, \ff)}
\newcommand{\Pfittzr}{\mathcal{P}(\tim, \xt, \ff_\tim)}
\newcommand{\Pfitttzzr}{\mathcal{P}(\tim + \Delta \tim, \xzrtt, \ff_{\tim + \Delta \tim})}
\newcommand{\Pfitzrtilde}{\mathcal{P}(\tim, \xt, \tilde{\ff}_\tim)}
\newcommand{\Pfitttzzrminus}{\mathcal{P}(\tim - \Delta \tim, x_{\tim - \Delta \tim}, \ff_{\tim - \Delta \tim})}

\newcommand{\PFitzr}{\mathcal{P}(\tim, \xt, \FF)}

\newcommand{\PFittzzr}{\mathcal{P}(\tim + \Delta \tim, \xzrtt, \FF)}

\newcommand{\safex}{x^\mathrm{safe}}

\newcommand{\xt}{x_t}
\newcommand{\xzrtt}{x_{\tim + \Delta \tim}}

\newcommand{\niinit}{\underline{n}_i}
\newcommand{\liinit}{\underline{l}_i}
\newcommand{\wiinit}{\underline{w}_i}

\newcommand{\yiarg}[1]{y_i[#1]}
\newcommand{\yjarg}[1]{y_j[#1]}
\newcommand{\yit}{y_{i}[\stag]}
\newcommand{\yiz}{y_{i}[0]}
\newcommand{\yio}{y_{i}[1]}
\newcommand{\yitm}{y_{i}[\stag-1]}

\newcommand{\eit}{e_i[\stag]}
\newcommand{\et}{e[\stag]}
\newcommand{\eizr}{e_i[0]}

\newcommand{\eio}{e_i[1]}
\newcommand{\eiTm}{e_i[\Fi-1]}

\newcommand{\nit}{n_i[\stag]}
\renewcommand{\ni}[1]{n_i[#1]}
\newcommand{\nt}{n[\stag]}
\newcommand{\nizr}{n_i[0]}

\newcommand{\nio}{n_i[1]}
\newcommand{\niT}{n_i[\Fi]}

\newcommand{\safen}[1]{n^\mathrm{safe}_{#1}}

\newcommand{\FF}{F}
\newcommand{\ff}{f}
\newcommand{\ffi}{f_i}

\newcommand{\tie}{\stag_i[e]}

\newcommand{\te}{\stag[e]}
\newcommand{\tia}[1]{\stag_i[#1]}
\newcommand{\tin}{\tia{n}}

\newcommand{\tn}{\stag[n]}

\newcommand{\Fi}{F_i}

\newcommand{\Teit}{\tau^\mathrm{edge}_{i,e[\stag]}}
\newcommand{\Teiz}{\tau^\mathrm{edge}_{i,e[0]}}
\newcommand{\Tei}{\tau^\mathrm{edge}_{i,e}}

\newcommand{\Tej}{\tau^\mathrm{edge}_{j,e}}
\newcommand{\Teitm}{\tau^\mathrm{edge}_{i,e[\stag-1]}}

\newcommand{\Ei}{e_i}

\newcommand{\Ej}{e_j}
\newcommand{\Ni}{n_i}
\newcommand{\Niarg}[1]{n_{#1}}
\newcommand{\Nj}{n_j}

\newcommand{\zedgearg}[1]{z^{\mathrm{edge}}_{#1}}
\newcommand{\znodearg}[1]{z^{\mathrm{node}}_{#1}}
\newcommand{\zslotarg}[1]{z^{\mathrm{slot}}_{#1}}
\newcommand{\zedge}{z^{\mathrm{edge}}}
\newcommand{\znode}{z^{\mathrm{node}}}
\newcommand{\zslot}{z^{\mathrm{slot}}}

\newtheorem{theorem}{Theorem}

\newtheorem{proposition}{Proposition}
\newtheorem{remark}{Remark}

\newtheorem{definition}{Definition}
\newtheorem{example}{Example}

\begin{document}

\maketitle

\begin{abstract}
In this paper, we investigate dynamic traffic optimization in railway systems, i.e., the behavior of these systems through time when their movements are dictated by solutions to optimization models with finite horizons. As interactions between trains are not considered beyond the limits of finite horizons, the danger of leading the system into a deadlock arises. In this paper we present new procedures to establish finite prediction horizons that are formally guaranteed to operate the system in a way that is compatible with the physical constraints of the network while avoiding deadlocking and minimizing computations. The key to this result is the notion of \textit{recursive feasibility}. This paper introduces conditions sufficient to attain it. We then discuss several important ramifications of recursive feasibility that enable efficient computations. We examine the possibility of  decomposing the underlying optimization models into smaller models with shorter horizons, or into models that only consider subsets of all trains. We also discuss warm starting and anytime approaches. We finally perform numerical experiments verifying these results on models that include a real-world railway system used for freight transport. On harder instances, some of our approaches outperform solving the same models as monolithic MILPs by more than two order of magnitude in terms of median computation times, while also achieving better worst--case optimality gaps.


\end{abstract}

\section{Introduction}

\begin{table}[!htbp]
	\centering
	\begin{tabular}{p{0.22\linewidth}p{0.67\linewidth}}
		\hline
		Symbol & Meaning\\
		\hline
		$i \in I$ & Reference to train $i$ in the set of all trains $I$\\
		$l \in L_n$ & Reference to the $l$-th slot at node $n$\\
		$\Ni, \Ei$ & List of nodes and edges on the path of train $i$ to its destination\\
		$\xt$ & State of the system at time $t$\\
		$\wiinit \in  \left[0,1\right]$ & Measured fraction of the first edge in train $i$'s path that has already been traversed at time $t$. It is $0$ if the train is dwelling at a node\\
		$\tie \in \mathbb{N}$ & Index of edge $e$ in $\Ei$\\
		$\tin \in \mathbb{N}$ & Index of node $n$ in $\Ni$\\
		$\Fi \in \mathbb{N}$ & Index of the terminal node for train $i$ in $\Ni$\\
		$\FF$ & Stacked list of number of nodes to terminal for each train, $\FF \doteq \left(\Fi\right)_{i \in I}$\\
		$\ffi \in \mathbb{N}$ & Number of nodes in the optimization horizon of train $i$\\
		$\ff_{i,t} \in \mathbb{N}$ & Number of nodes in the horizon of train $i$ in the optimization model constructed at $t$\\
		$\ff_t$ & Stacking of the optimization horizons into a list, $\ff_\tim \doteq \left(\ff_{i,\tim}\right)_{i \in I}$\\
		$\tau^{\mathrm{edge}}_{i,e} \in \mathbb{R}_+$ & Travel time of train $i$ on edge $e$\\
		$\Pfittzr$ & Optimization model constructed at time $t$, when the state of the system is $\xt$ and under the optimization horizons determined by $\ff_t$
		\\
		\hline
		Opt. Variable & Meaning\\
		\hline
		$\yit \in \mathbb{R}_+$ & Departure time of the $k$-th transit stage for train $i$ \\
		$\zedgearg{i,j,e} \in \left\{0,1\right\}$ & Is $1$ if train $i$ transits over edge $e$ before train $j$, $0$ otherwise\\
		$\znodearg{i,j,n} \in \left\{0,1\right\}$ & Is $1$ if train $i$ transits over node $n$ before train $j$, $0$ otherwise\\
		$\zslotarg{i,n,l} \in \left\{0,1\right\}$ & Is $1$ if train $i$ occupies slot $l$ at node $n$, $0$ otherwise\\
	\end{tabular}
	\caption{Summary of symbols used in the paper.}
	\label{tab:list_of_symbols}
\end{table}




The optimization of schedules for railway systems has attracted a significant amount of research over the decades. Since many of the models in the literature entail decisions that are logical in nature (e.g., which train has precedence) the corresponding optimization models are typically NP-hard combinatorial or mixed--integer~\cite{laumanns_train, dariano_bandb, dariano_macroscopic, deshutter_mpc_railway}. The computations required to solve these models are often a significant impediment to their practical applicability, in particular on extensive networks involving large train fleets.


To allow computations, these models typically only consider finite optimization horizons. Since trains are unaware of each other beyond the limits of these optimization horizons, however, the risk of driving the system into a deadlock\footnote{A deadlock is a situation where trains cannot pass or reach their destination without at least one reversing. A deadlock can occur well before trains meet, if the set of forward-moving actions will always ultimately result in them meeting and being unable to pass; see Section~\ref{sec:recursive_feasibility}.} arises if the terminal state of the trains does not satisfy certain provisions. A primary goal of this paper is a thorough analysis of this issue. Borrowing notions from control systems theory, more precisely model predictive control~\cite{morari_mpc_book}, we establish conditions sufficient to guarantee a deadlock--free operation of the network, when it is controlled through time by the results of traffic scheduling models with finite horizons. We present an algorithm to dynamically dimension optimization horizons to achieve this. 

The key idea developed in this paper is that of a \emph{safe state}, and how it is a sufficient condition to guarantee \textit{recursive feasibility}. After presenting an optimization model for scheduling train traffic in Section~\ref{sec:opt_model_static}, we discuss in Section~\ref{sec:closed_loop_control} how recursive feasibility, in turn, guarantees deadlock--free operations when trains transit according to schedules that are updated through time, as they move, and with horizons stretching further and further into the future. This guarantee does not require solutions' optimality, meaning that computations can be safely interrupted as soon as a first feasible solution is obtained. As a by--product we derive an algorithm for early detection of deadlocks.


In Section~\ref{sec:computational_ramifications_rf} we discuss several important by-products of this analysis, that are of interest when addressing the computational challenges mentioned above. Further, we discuss methods to divide and solve optimization programs with longer horizons as a sequence of optimization models with shorter horizons. Additionally, we explore the possibility of warm-starting the optimization procedures when previous solutions are available; we demonstrate with a counterexample that this is not generally possible, as previous solutions can become invalid as time passes even in the absence of disturbances and with perfect information of travel times. We provide conditions to avoid these issues and ensure successful warm-starting and analyze how these ideas can be used as a basis for the development of anytime solution strategies. Last, we show how the results presented in the paper enable solution procedures that only consider arbitrary subsets of the train fleet at a time.

To the best of the authors' knowledge, the notion of safe states and recursive feasibility presented in Section~\ref{sec:closed_loop_control}, as well as their application to the computational techniques introduced in Section~\ref{sec:computational_ramifications_rf}, are new in the context of railway systems.

Finally, in Section~\ref{sec:results} we discuss computational experiments involving both a synthetic as well as a real--world railway network used for freight transport. For the harder instances involving the real--world network model, median compute times of some decomposition approaches presented in this paper surpass solving the same optimization models monolithically as MILPs using Gurobi by more than two orders of magnitude. Our approaches also achieve superior worst case optimality gaps for these harder instances.

\section{Literature Review}

Model predictive controllers (MPC) have been used for decades to control slow plants, such as chemical processing plants~\cite{morari_ppf}. Owing to the advances in the performance of modern computers, in the recent decade its application has been extended to the control of fast real--time systems, including power converters~\cite{mariethoz_fast_mpc_pc}, robotic applications~\cite{robots_mpc} and autonomous vehicles~\cite{mpc_autonomous_vehicles}. The application of MPC to railway systems has also been investigated in~\cite{laumanns_train}, which concentrates on the operation of a central railway station area, as well as in~\cite{deschutter_mpc_rail} where the authors examine the application of linear complementarity problems to describe the system. The stability for MPC is well understood~\cite{morari_mpc_book, rawlings2017model}, and relies on the construction of appropriate terminal sets or costs. For linear, time--invariant systems subject to linear constraints, systematic procedures have been developed to compute these terminal sets and costs~\cite{morari_mpc_book}. However, no general result is available to achieve the same guarantees for systems entailing logic decisions. They usually have to be derived and tailored for the specific application at hand. One example of this is in~\cite{recursive_feas_airways}, where the authors establish appropriate terminal sets in the context of airway trajectory planning.

The literature on optimization models for railway systems is vast. The survey in~\cite{tornquist2007n} categorizes papers based on whether they only consider railway lines or generic networks, whether non-station segments can have bi-directional or only uni-directional traffic, the maximum number of tracks per segment (single, double or $N$ tracked segments) and whether meet points and stations are assumed to have finite or infinite capacity. The MILP model presented in this paper is similar to the one introduced in~\cite{tornquist2007n}, allowing for single and double tracks, bi-directional traffic, and the capacity at intermediate stations, determined by the number of available parallel tracks, is explicitly taken into account. 

In~\cite{goverde2007railway, goverde2010delay}, the authors look at the application of max-plus algebras and max-plus linear systems to achieve macroscopic models of delay propagation, which are then incorporated within optimization programs. The work in~\cite{laumanns_train} proposes a dispatching assistant tool based on a linear, binary optimization model that assigns precomputed blocking-stairways to trains as a mechanism to reserve track resources and concurrently satisfy connection and platform related constraints while maximizing customer satisfaction. In~\cite{suhl2001design}, the authors present different approaches for modeling waiting policies, ranging from methods based on exact optimization to macroscopic, agent-based simulations, without, however, considering capacity constraints. The work in~\cite{dariano_bandb, dariano_macroscopic} introduces alternative arc graphs to model and optimize train schedules. This abstraction maps the current state of the railway system and the trajectories of the trains into a graph that represents all conflicts related to the precedence of resource usage. These reflect the temporal inter dependencies of the operations that have been executed (segment traversal, stops at stations, etc.) for the trains to get to their destinations. 
 
A related approach is in~\cite{weymann2011qualitat}, in which decision variables are introduced to encode train sequencing but the continuous part of the model includes variable speed profiles. The main potential drawback of this approach is the required growing number of optimization variables. 
The work in~\cite{schlechte2011micro} analyzes approaches to systematically simplify excessively detailed network models for timetabling purposes. Other approaches that deal with computational aspects of solving macroscopic models are in~\cite{norio2005train}, which uses simulated annealing, and~\cite{tornquist2007n} who implement several heuristics, depending on the size of the instances, including computationally cheap greedy strategies. Similar to our paper, the research presented in~\cite{tornquist_kraseman} is also focused on reducing computational complexity, although their work concerns rescheduling problems after traffic perturbations. The scheme proposed is based on fast heuristics to determine prioritizations, and the authors tackle issues related to deadlocking by implementing a backtracking strategy in their search tree that changes the value of previously set binary decisions when they are found to lead to inconsistencies. Our approach constructs solutions in the sub problems that are guaranteed to avoid deadlocking, thus avoiding the necessity of performing `backtracking' as we solve optimization models through time.

Previous research on decomposition techniques applied to railway systems optimization is also present. The work in~\cite{kersbergen} discusses an application of distributed model predictive control. The scheme proposed results in a geographical decomposition, in contrast to our work which separates the problem in time or by trains. Additionally, their approach hinges on reordering the constraint matrix of the model using an MIQP. The objective of this reordering is to obtain a constraint matrix that is as close as possible to block-diagonal; for this scheme to work it is necessary that the coupling constraints outside of the blocks in the diagonal only entail continuous decisions. We do not suffer from this type of restriction -- feasibility is guaranteed through boundary conditions, while couplings between different sub-problems can be arbitrarily complex. Further, while their method can encourage the reordering to lead to a balanced block-diagonal matrix, the extent to which this is possible is not fully under control. In contrast, our decomposition scheme allows for full controllability of the complexity of the sub-problems: we can, for instance, arbitrarily decide how many trains we want to consider in each step in our proposed train-wise decomposition, see Section \ref{subsec:train-wise-decomp}.
	
In~\cite{pellegrini_1}, the effects of different representation granularities are investigated, along with some discussion on rolling horizon optimization. In this work, optimizations are split through time in `windows', in contrast to our approach of prescribing different optimization horizons for each train. The limitations of these `windows' are acknowledged, requiring trains or decisions outside these `windows' to be considered to avoid deadlocking. It is unclear whether such measures to modify the planning horizon guarantees deadlock-free plans. In this paper, we examine this problem, by way of independent horizons for each train, and illustrate there can be multiple solutions, see Section \ref{sec:recursive_feasibility}.
	
The work in~\cite{lamorgese} presents a model decomposition technique that is similar to traditional Benders' decomposition. Coherence among the sub-problems is guaranteed by an iterative master-slave scheme and message passing, while our work provides \emph{a priori} guarantees that sub-problems can be reassembled into a feasible global solution.
	
In addition, one of the core results of this paper, which allows the models to be split into train sub-fleets of arbitrary size, has not been discussed in any of these prior works on decomposition techniques.

\section{Railway System Description and Traffic Optimization Model}
\label{sec:opt_model_static}


In this work we consider railway systems comprised of stations connected by single or double tracks. There can be junctions and meet points on the lines connecting these stations. We represent this as a graph with nodes and edges. 

Edges represent the tracks required for the motion of the train from one node to the next to take place. We allow for single or double edges to model, respectively, single or double tracks. A single edge can be utilized by only one train at the time. We further assume that double tracks are utilized exclusively in bidirectional mode, i.e., we do not allow two trains transiting in the same direction to utilize the two lines of a double track at the same time.

Nodes characterize the only locations \textit{in the schedule} where train stops can be planned -- once a train starts transiting over an edge, it can complete its motion without interruption until it reaches the next node. Trains may stop in the middle of an edge as they \textit{physically} transit through the network, e.g., due to disturbances, but the schedules we generate are constructed ensuring that, nominally, trains can travel without those stops. To emphasize its nature as an uninterrupted process, we refer to the motion of a train from a node to the next as a transit \textit{stage}. In the models discussed in the experimental section~\ref{sec:results}, which include a real freight network, we have placed nodes at locations where signaling is present. While a train is stopped, it still occupies tracks: we represent this as a train occupying a \emph{slot} inside a node. Stations as well as meet points are simplified into single nodes with multiple slots, meaning that the additional maneuvers and time required for a train to transit through or stop at a station are neglected, but if a train does dwell at, e.g., a station, then that is represented as the train dwelling at the node and occupying a slot for a corresponding period of time.

We finally assume, for technical simplicity, that the path of a train from its current position to its destination has been predetermined. It is possible to add path selection to the optimization model presented in this work using the approach introduced in~\cite{dariano_path_opt}, although this is outside the scope of this paper.

The purpose of the optimization model presented in this section is to produce movement schedules over a finite horizon after sampling the current state of the railway system at a given point in time. These computations might be required if, for instance, a previous schedule is invalidated by perturbations, or if it has been executed to completion. The objective pursued is the maximization of the network's throughput, proxied in our objective by the minimization of the aggregated arrival time of all trains to their final nodes. This is a model suitable for cases in which the railway system is used for freight transportation~\cite{huntervalley_train, railfreight}. Note in particular that we do not have any requirement related to adherence to preexisting timetables. We briefly discuss how the model presented in this paper could be extended to incorporate these. But to keep the discussion centered on the characterization of deadlocks, our model focuses on the \emph{physical} constraints on railway traffic, rather than the \emph{operational} ones. We finally observe that in Section~\ref{sec:closed_loop_control} we will argue that infeasibility challenges, which represent the starting point for the main contribution of this paper, are independent of the specific model used to represent the system discussed in this section.

\subsection{Optimization Model Formulation}

Given a graph $\mathcal{G}(\mathcal{E},\mathcal{N})$ comprising a set of edges $\mathcal{E}$ and nodes $\mathcal{N}$, and a set of trains indexed by $i \in I$. For each train $i \in I$ we introduce
\begin{eqnarray}
\begin{array}{l}
	\Ni = \left( \nizr, \nio, \dots, \niT \right) \\ 
	\Ei = \left( \eizr, \eio, \dots, \eiTm \right)
\end{array}
\label{eq:nodes_and_edges_in_path}
\end{eqnarray} 
as the sequences of nodes and edges, respectively, in the path of train $i$ from its current position to its destination node $\niT$, where $\Fi$ characterizes the number of stages to the terminals. Note that we always use bracket notation $[\cdot]$ when we refer to stages. If a train is currently transiting an edge, then that edge is $\eizr$ in $\Ei$, and $\nizr$ is the last node it visited. 
Let $\tie$ be the index of edge $e$ in $\Ei$, and $\tin$ be the index of node $n$ in $\Ni$. Whenever clear from the context, we drop the index $i$ and have $\te, \tn, \nt, \et$ instead.

\textbf{Sequentiality of transit.} The initial set of constraints represents the required temporal sequentiality of transit over the edges of the network. Let $\yit \in \mathbb{R}^+$ be the optimization variable modeling the time at which train $i \in I$ \textit{departs} from the $\stag-$th node $\nit$. Then,
\begin{align}
\begin{array}{lll}
	 \yio \geq \yiz + \Teiz \cdot (1-\wiinit) & \forall i \in I\\
	 \yit \geq \yitm  + \Teitm  & \forall i \in I, \ \stag = 2, \dots, \Fi -1,
\end{array}
\label{eq:sequentiality}
\end{align}
where $\Teit \in \mathbb{R}$ is the time required by train $i$ to complete travel over the $\stag$--th edge $\eit$, which for the first stage is reduced by the fraction of the edge already traversed $\wiinit$. 
The underlining of $\wiinit$ indicates that this is a measurement of the current state of the system used to initialize the optimization model, an aspect which will be further analyzed in Section~\ref{sec:closed_loop_control} where we discuss closed--loop operation. Note that $\Teit$ can depend on several factors, including the current speed of the train, its state, i.e., whether it's empty or loaded with goods, wear and tear conditions, its length and number of locomotives, etc. As long as these characteristics are effectively captured in the edges travel times they fit into our optimization framework.

\emph{Initial conditions.} Let $I^\mathrm{edge} \subseteq I$ be the subset of trains currently transiting an edge (i.e., not stopped at a node), and $\eizr$ be that edge. Then we have
\begin{eqnarray}
\begin{array}{ll}
\yiz = 0 & \forall i \in I^\mathrm{edge}.
\end{array}
\label{eq:init_for_y0}
\end{eqnarray} 

\textbf{Edge conflicts.} We partition the set of edges into single and double tracks, $\mathcal{E} \doteq \mathcal{E}^\mathrm{s} \cup \mathcal{E}^\mathrm{d}$. The former allows the transit of at most one train at the time, while on the latter two trains can transit as long as they are headed in opposite directions. For each single track edge $e \in \mathcal{E}^\mathrm{s}$ we form the set of conflicts
\begin{eqnarray}
C_e \doteq \left\{\left. (i,j) \ \forall i,j \in I, \ j > i \ \right| e \in \Ei, \ e \in \Ej \right\},
\end{eqnarray}
encoding the fact that if both trains $i$ and $j$ are to transit over edge $e$ within their planned path to a destination, then a conflict must be resolved to determine the train transiting first. The construction for double  edges $e \in \mathcal{E}^\mathrm{d}$ is similar, but conflicts are considered only among trains transiting in the same direction. Then we have
\begin{eqnarray}
\begin{array}{lll}
\yiarg{\te} \geq \yjarg{\te} + \Tej - M \zedgearg{i,j,e} 
\\  
\yjarg{\te} \geq \yiarg{\te} + \Tei - M(1-\zedgearg{i,j,e}) 
\\
\zedgearg{i,j,e} \in \left\{0,1\right\}, 
\end{array}
\label{eq:edge_conflicts}
\end{eqnarray}
for all $(i, j) \in C_e, e \in \mathcal{E}$, where we introduced the binary optimization variable $\zedgearg{i,j,e}$ which attains $1$ if train $i$ is scheduled to transit before $j$ over edge $e$, and is $0$ otherwise. The value of $M$ has to be set sufficiently large, e.g.,~$M \geq \max_{i \in I} \yiarg{\Fi-1}$.

%

\emph{Initial conditions.} Trains that are currently transiting an edge $i \in I^\mathrm{edge}$ automatically get priority over that edge:
\begin{eqnarray}
\begin{array}{ll}
	\zedgearg{i,j,\eizr} = 1 & \forall i \in I^\mathrm{edge}, \ (i,j) \in C_{\eizr}.
\end{array}
\label{eq:init_for_xedge}
\end{eqnarray} 

\textbf{Node conflicts.} Similar to edges, the resolution of conflicts over a node involves deciding which train transits first, and is encoded with the binary variable $\znodearg{i,j,n}$, attaining $1$ if train $i$ transits over $n$ before train $j$. Nodes are characterized by a number of ``slots'' indicating how many trains can be present at that node at the same time. Before transiting, a train thus also needs to acquire a slot on the nodes along its path. To capture this, we introduce the binary variable $\zslotarg{i,n,l}$ which indicates whether train $i$ occupies slot $l \in L_n$ on its transit over node $n$, where $L_n$ is the set of slots at node $n$. Then, for each $n \in \mathcal{N}$, we introduce the following set to capture conflicts over nodes:
\begin{eqnarray*}
	C_n \doteq \left\{ (i,j) \ \forall i,j \in I, j>i \left|\right. n \in \Ni, \mathrm{ \ and \ } n \in \Nj \right\},
\end{eqnarray*}
and require that schedules satisfy the following constraints:
\begin{eqnarray}
\begin{array}{llr}
\yjarg{\tn-1} \geq& \yiarg{\tn} - M (1-\znodearg{i,j,n})\\
& - M (1-\zslotarg{i,n,l}) - M (1- \zslotarg{j,n,l}),\\ 
\yiarg{\tn-1} \geq& \yjarg{\tn} - M \znodearg{i,j,n}\\
&- M (1- \zslotarg{i,n,l}) - M (1- \zslotarg{j,n,l})\\ 
\end{array}
\label{eq:node_conflicts}
\end{eqnarray}
for all $n \in \mathcal{N}, l \in L_n$, and $(i,j) \in C_n$. These constraints can be active only if, for a given node $n$ and slot $l$, both $\zslotarg{i,n,l}$ and $\zslotarg{j,n,l}$ attain a value of 1 in the solution, i.e., both trains are scheduled to use the same slot during their transit. In such a case, the constraint ensures that if train $i$ transits before $j$ on the node, then the start time of train $j$ over the edge \textit{leading} to node $n$ has to be greater or equal to the start time of $i$ \textit{leaving} node $n$. Additionally, each train occupies exactly one slot during transit:
\begin{eqnarray}
\sum_{l \in L_n} \zslotarg{i,l} = 1 & \forall i \in I.
\label{eq:node_conflict_ii}
\end{eqnarray}

\emph{Initial conditions.} As with edges, trains that are currently transiting a node are occupying a slot  on that node and, hence, they automatically get priority over that node and acquire a slot. Let $I^\mathrm{node} \subseteq I$ be the subset of trains currently transiting on a node, $\nizr$ be that node and $\liinit$ be the slot they are currently occupying. Then,
\begin{eqnarray}
\begin{array}{ll}
\znodearg{i,j,\nizr} = 1 & \forall i \in I^\mathrm{node}, \ (i,j) \in C_{\nizr}\\
\zslotarg{i,\nizr,\liinit} = 1 & \forall i \in I^\mathrm{node},
\end{array}
\label{eq:init_for_nodes}
\end{eqnarray}
where, as before, the underlining of $\liinit$ indicates that this is part of the state that is measured.

\textbf{Objective function.} As a proxy for rail network throughput maximization, the objective we pursue is the minimization of the sum of the trains' arrival times,
\begin{eqnarray}
\begin{array}{ll}
	\mathrm{minimize} & \sum\limits_{i \in I} y_{i, \Fi-1}.
\end{array}
\label{eq:obj_function}
\end{eqnarray}
We note that there is significant flexibility in the type of objectives that could be used including, for example, penalties on delays at intermediate steps. To achieve this, for some train $i \in I$ scheduled to depart from stage $k$ at the reference time $y_i^\mathrm{ref}[k]$ stemming from, e.g., a timetable, a new optimization variable $y^{\mathrm{dev}}_i[k] \geq 0$ can be introduced as
\begin{eqnarray*}
	- y^{\mathrm{dev}}_i[k] \leq y_i[k] - y_i^\mathrm{ref}[k] \leq y^{\mathrm{dev}}_i[k].
\end{eqnarray*}
This variable could then be added to the objective function allowing a straightforward extension of the model presented to pursue timetable adherence.

In summary, the complete model considered in this paper is
\begin{eqnarray}
\Pp: \left\{
\begin{array}{ll}
	\min\limits_{y, z} & \text{total completion time } \eqref{eq:obj_function}\\
	\mathrm{s.t.} & \text{sequentiality } \eqref{eq:sequentiality}\\
	& \text{edge conflicts } \eqref{eq:edge_conflicts}\\
	& \text{node conflicts } \eqref{eq:node_conflicts},\eqref{eq:node_conflict_ii}\\
	& \text{initial conditions } \eqref{eq:init_for_y0}, \eqref{eq:init_for_xedge}, \eqref{eq:init_for_nodes}\\
	& \zedge, \znode, \zslot \in \left\{0, 1\right\}\\
	& y \geq 0.\\
\end{array} 
\right.
\end{eqnarray}

\section{Receding Horizon Control and Recursive Feasibility}
\label{sec:closed_loop_control}




In this section we embed the optimization model of Section~\ref{sec:opt_model_static} within a strategy called receding horizon control, in which the optimization horizon for train $i \in I$ is shortened to \hbox{$0 \leq \ffi \leq \Fi$} in \eqref{eq:nodes_and_edges_in_path}. 

Within this framework, feedback is introduced by continuously measuring the current state of the system through time, denoted by $\xt$, and using this new information to compute a new solution to $\Pp$. The state of the system $\xt \doteq (\niinit, \wiinit, \liinit)_{i \in I}$ denotes the complete set of measurements required to initialize the optimization model $\Pp$, where $\niinit \equiv \nizr \in \mathcal{N}$ is the most recent node visited by train $i \in I$, \hbox{$0 \leq \wiinit \leq 1$} is the fraction of the edge $\eizr$ already traversed and $\liinit$ indicates the slot occupied if the train is  currently located at a node. We indicate with $\Pfittzr$ the instance of $\Pp$ generated at time $t$ for the initial state $\xt$ and under the optimization horizon $\ff_\tim = \left(\ff_{i,\tim}\right)_{i \in I}$. In this section we thus examine the evolution through time of the state of the railway system $\xt$ under the control of movement schedules produced by $\Pfittzr$. 

For simplicity of notation and exposition, we assume here that schedules are recomputed at constant intervals of time $\Delta t$, even though this is not necessary and the results presented in this paper can be applied verbatim in contexts where arbitrary events, such as trains arriving late at a node, are used to trigger plan recomputations. Note that $t$ represents global (continuous) time, while $k$ in the previous section was an index of time expressed as an integer number of stages relative to the position of the system at the time it was instantiated.

The potential problem with shortening prediction horizons is that it makes the controller ``blind'' to trains interactions in the later stages, which might lead to deadlocking.

\begin{example}
	\label{example:deadlocking}
	Consider the state of the system depicted in Figure~\ref{fig:example_base}: $T_1$ and $T_2$ originating from separate branches of the network are about to merge on the same single line with two passing sidetracks, while $T_3$ is transiting in opposite direction. The destination nodes for the trains are indicated with grayed out, dotted arrows: the destination for $T_1$ and $T_2$ is $n_5$ (which could represent a station), while the destination for $T_3$ is $n_0$, resulting in
		\begin{eqnarray*}
			\begin{array}{ll}
				n_{T_1} = (n_1, n_3, n_4, n_5), & e_{T_1} = (e_{1-3}, e_{3-4}, e_{4-5}),\\
				n_{T_2} = (n_2, n_3, n_4, n_5), & e_{T_2} = (e_{2-3}, e_{3-4}, e_{4-5}),\\
				n_{T_3} = (n_7, n_6, \dots, n_0), & e_{T_3} = (e_{7-6}, \dots, e_{0-1}).
			\end{array}
		\end{eqnarray*}
	Trains are stopped with their heads at the end of the blocks on which they are dwelling (indicated with red squares, for ``stopped''). 
	
	Assume that at that point in time $\tim$, the optimization horizons for the individual trains used to construct model $\Pfittzr$ are as depicted in Figure~\ref{fig:example_deadlocking}, i.e.,
	\begin{eqnarray*}
	f_t = \left[\begin{array}{c}
	f_{T_1}\\
	f_{T_2}\\
	f_{T_3}
	\end{array}\right] = \left[\begin{array}{c}
	3\\
	2\\
	2
	\end{array}\right].
	\end{eqnarray*}
	A feasible solution to $\Pp$ for this situation is shown in the train graph of Figure~\ref{fig:train_graph_deadlocking}. As trains transit according to this feasible schedule, the system becomes physically deadlocked at $\tim + \Delta \tim$. This is also reflected by the optimization model $\PFittzzr$ becoming infeasible. It should be emphasized that the challenge illustrated in this example is a result of the physical state of the system represented in Figure~\ref{fig:example_base} and the extent of the horizons within which trains `see' each other. The specific optimization model described in Section~\ref{sec:opt_model_static} becoming infeasible is only a reflection of the system becoming deadlocked, implying that other optimization models would be subject to the same challenge under these conditions. Note also that this occurs despite the absence of imperfect information or system disturbances, i.e., we have perfect knowledge of travel times and they don't deviate from their nominal values. 
\end{example}

\begin{figure}
	\centering
	\includegraphics[width=1\linewidth]{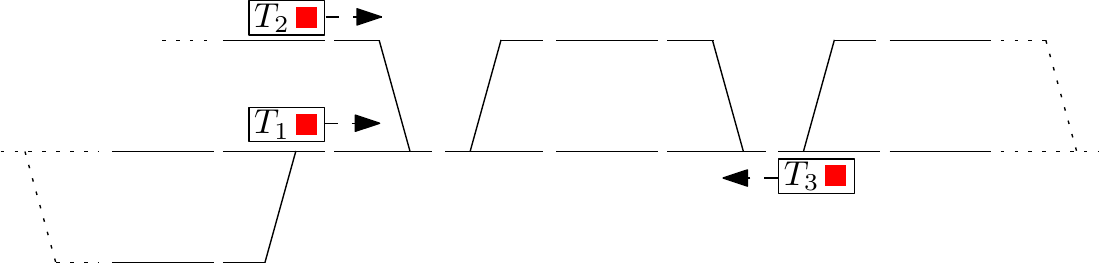}
	\caption{System state discussed in Example~\ref{example:deadlocking}.}
	\label{fig:example_base}
\end{figure}
\begin{figure}
	\centering
	\includegraphics[width=1\linewidth]{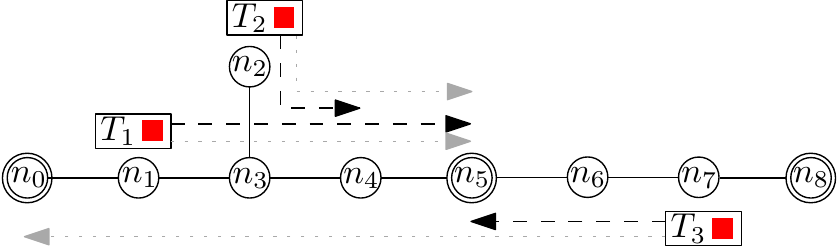}
	\caption{Graph model for the system discussed in Example~\ref{example:deadlocking}.}
	\label{fig:example_deadlocking}
\end{figure}
\begin{figure}
	\centering
	\includegraphics[width=0.6\linewidth]{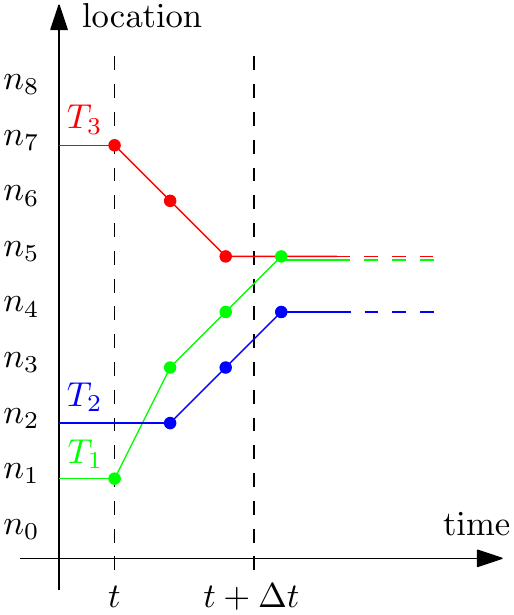}
	\caption{Train graph discussed in Example~\ref{example:deadlocking}.}
	\label{fig:train_graph_deadlocking}
\end{figure}

Instances affected by a deadlock are reflected as models that do not allow for a finite, feasible set of start times $y$, i.e., $\PFitzr$ is infeasible. Given that $\PFitzr$ exclusively entails physical constraints on traffic, rather than operational ones such as deadlines, an infeasible model indicates that there is no sequence of decisions steering trains from their current position to their respective terminals that is compatible with the physical limitations on traffic, i.e., that there is a deadlock. Hence, a state $\xt$ is deadlocked if and only if $\PFitzr$ is infeasible.

In the following section we take a closer look at the relationship between $\PFitzr$ and $\Pfittzr$ in the context of deadlocking.




\subsection{Recursive Feasibility}
\label{sec:recursive_feasibility}


Recursive feasibility is the fundamental notion used to establish the stability of linear, time-invariant systems under receding horizon controllers such as model predictive controllers~\cite{camacho_mpc}. Even though our system is neither, due to the presence of binary variables and the fact that the constraints are time--varying, the issue of recursive feasibility remains crucial in ensuring that the system is not driven into a deadlocked state when the prediction horizons are shortened to $0 \leq \ff \leq \FF$. To simplify the exposition and the proofs, we assume infinite capacity at the terminals. This allows us to present core ideas without incurring in the technical difficulties encountered when dealing with infinite horizons~\cite{morari_mpc_book}.

The core notion introduced in this paper to ensure recursive feasibility is that of a safe state.

\begin{definition}[Safe state]
	\label{def:safe_state}
	A \textit{safe state} is a system state \hbox{$\safex  \doteq (\safen{i}, 0, l_i^\mathrm{safe})_{i \in I}$} in which all trains are at a node, and all nodes $n \in \mathcal{N}$ in the graph have an unoccupied slot.
\end{definition}

\begin{definition}[Non--regressiveness]
	\label{def:non_regressiveness}
	We define as \emph{non--regressive with respect to $x$} a system state in which trains occupy nodes that are successors along their paths from a given state $x$.
\end{definition}

We overload the inequality sign ``$\leq$'' when applied to horizons to indicate non-regressiveness: $\ffi \leq \tilde{\ffi}$ means that, for train $i$, the horizon determined by $\tilde{\ffi}$ terminates at a node that is further along $i$'s path than the node reached by $\ffi$. When $\ffi$ and $\tilde{\ffi}$ refer to two different points in time, the numbers might not satisfy the standard meaning of the inequality, but they still do imply non-regressiveness.

The following result is a characteristic of safe states which will be used to prove recursive feasibility.

\begin{proposition}
	There always exists a sequence of train movements that drives the system from any safe state $\safex_a$ into any other safe state $\safex_b$ that is non-regressive with respect to $\safex_a$.
\end{proposition}

\begin{proof}
	Algorithm~\ref{alg:trivial_safe_to_safe} constructs one such sequence of movements. Since the initial state is safe, any train can be moved forward to any other node in the network in a first step; the destination node has to have at least two slots (otherwise it cannot be part of a safe state). Upon train arrival, the node has now either no empty slots left or at least one. If it still has at least one empty slot, then the current state is also safe, and the procedure can restart by picking any other train that hasn't been moved yet. If the current node has no slots left, there must be another train on the current node that has not been moved yet. By construction, all other nodes have at least one empty slot available for transit, meaning that the train can be moved anywhere in the network. This procedure can be repeated to termination.
\end{proof}


\begin{algorithm}
	\caption{Trivial Policy From Safe State to Safe State}
	\label{alg:trivial_safe_to_safe}
	\begin{algorithmic}[1]  
		\State $I^{\mathrm{open}} \gets I$
		\State $i \gets$ random choice from $I^{\mathrm{open}}$
		\While{$I^{\mathrm{open}}$ not empty}
		\State let $i$ transit from $\safex_{i,a}$ to $\safex_{i,b}$
		\State remove $i$ from $I^{\mathrm{open}}$
		\If{$\safex_{i,b} = \safex_{j,a}$ for some $j \in I^{\mathrm{open}}$}
		$i \gets j$
		\Else{} $i \gets$ random choice from $I^{\mathrm{open}}$
		\EndIf
		\EndWhile
	\end{algorithmic}
\end{algorithm}


Algorithm~\ref{alg:horizon_computation} presents a procedure to compute a dynamic horizon $\ff_{\tim}$ based on the notion of safe states. If the system is in a non-deadlocked state $\xt$, it is guaranteed to successfully compute an optimization horizon $\ff_{\tim}$ which ensures recursive feasibility.
In the proposed procedure, the optimization horizon for each train $\ffi$ is iteratively extended until it reaches a node such that the state of the system would be safe if trains transited up to that point from their current position. Prediction horizons are further extended until the computed $\ff_{\tim}$ results in a feasible $\Pfittzr$, while retaining the condition on the final state being safe, a condition that is guaranteed to be met if $\PFitzr$ is feasible. We call horizons $\ff_{\tim}$ computed according to Algorithm~\ref{alg:horizon_computation} \textit{safe optimization horizons}.

\begin{remark}
	\label{remark:ftilda_geq_f_always_feasible}
	Note that the feasibility of $\Pfittzr$ implies the feasibility of $\Pfitzrtilde$ for any $\tilde{\ff}_\tim \geq \ff_\tim$ that leads to a safe state. This is true because Algorithm~\ref{alg:trivial_safe_to_safe} can always be used to generate a feasible schedule between the corresponding safe states. 
\end{remark}

Thus, choosing larger initialization horizons (line $1$ of the algorithm) reduces the number of models that have to be attempted before a feasible one is found. 

\begin{algorithm}
	\caption{Compute Safe  Horizon Termination $\ff$}
	\label{alg:horizon_computation}
	\begin{algorithmic}[1]  
		\State $\ffi \gets 1$, for all $i \in I$
		\State $\eta_n \gets |L_n|$, for all $n \in \mathcal{N}$
		\ForAll{$i \in I$}
		\While{$\eta_{\ni{\ffi}} \leq 1$}{}
		\State $\ffi \gets \ffi+1$
		\EndWhile
		\State $\eta_{\ni{\ffi}} \gets \eta_{\ni{\ffi}} - 1$
		\EndFor
		\If{$\Pfitzr$ is feasible} return $\ff$
		\Else
		\State select some $i \in I$ for which $\ffi < \Fi$
		\State $\ffi \gets \ffi+1$
		\State \textbf{goto} 2 
		\EndIf
	\end{algorithmic}
\end{algorithm}


We next prove the correctness of Algorithm~\ref{alg:horizon_computation}, and additionally provide a characterization of deadlocks that is generally computationally cheaper than solving the full-horizon model $\PFitzr$. 

\begin{theorem}[Deadlock characterization and recursive feasibility]
	\label{thm:recursive_feasiblity} 
	Let $\Pfittzr$ be the optimization program instance generated at time $t$ for the initial state $\xt$ and with any safe,  non-regressive optimization horizon $\ff_\tim$ produced by Algorithm~\ref{alg:horizon_computation}. Then, 
	\begin{enumerate}[a)]
		\item the state $\xt$ is not deadlocked if and only if $\Pfittzr$ is feasible, and
		\item if $\Pfittzr$ is feasible, then its operation is recursive feasible.
	\end{enumerate}
\end{theorem}





\begin{proof}
	We first demonstrate part b) of the Theorem. 

	Let $(\tilde{y}, \tilde{z})$ be any feasible solution of $\Pfittzr$, where $\ff_{\tim}$ is a safe  horizon computed at time $\tim$ according to Algorithm~\ref{alg:horizon_computation}. Let $(\overline{y}, \overline{z})$ be its augmentation to the terminal nodes by application of the trivial policy in Algorithm~\ref{alg:trivial_safe_to_safe}.
	Note that a state in which all trains are at terminals is always safe under our assumption of infinite capacity at these nodes. By construction of $\ff$ and Algorithm~\ref{alg:trivial_safe_to_safe}, $(\overline{y}, \overline{z})$ is feasible for $\PFitzr$ and, hence, $(\overline{y} - \Delta t, \overline{z})$ is feasible for $\PFittzzr$. This shows that $(\overline{y} - \Delta t, \overline{z})$ produces a path from $\xzrtt$ to the terminals going through the safe state configuration given by the horizon termination schedule $\ff_t$ computed at $t$. We can extend this part of the solution by applying Algorithm~\ref{alg:trivial_safe_to_safe} to construct a sequence of decisions to any non-regressive terminal conditions $\ff_{\tim + \Delta \tim}$ computed at $\tim+\Delta \tim$, ensuring the feasibility of $\Pfitttzzr$. This concludes the proof of part b).

	For part a), if $\Pfittzr$ is feasible, the system is not deadlocked since, as shown in part b), we can always construct a feasible solution to $\PFitzr$ by applying Algorithm~\ref{alg:trivial_safe_to_safe}. On the other hand, Algorithm~\ref{alg:horizon_computation} extends $\ff$ until $\Pfittzr$ is feasible, which is guaranteed to succeed if the system is not deadlocked.
\end{proof}

Note that with the application of Algorithm~\ref{alg:horizon_computation} we merely determine safe optimization horizons to use within the optimization model. The trivial policy described in Algorithm~\ref{alg:trivial_safe_to_safe} is not relevant at this stage: at some given point in time, the system will generally not be in a safe state, hence the policy is not applicable. Rather, an appropriate optimization model $\Pfittzr$ is solved to determine a movement schedule from the current generic state to a safe state. The trivial policy is then applicable to drive trains from that safe state into any other subsequent safe state. Having access to a policy that is guaranteed to keep trains moving after the end of the finite optimization horizon is what allows us to ensure that the system is not driven into a deadlock. Note that we will typically recompute a new sequence of controls, i.e., a new solution to $\Pp$ for the new state and new horizons, before any of the trains have arrived at the final node within their respective horizons, in which case the system will generally not traverse that safe state. Also, note that the optimization horizons required by Theorem~\ref{thm:recursive_feasiblity} are not unique: in Example~\ref{example:deadlocking}, both $\left\{T_1: n_5, T_2: n_8, T_3: n_0\right\}$ as well as $\left\{T_1: n_8, T_2: n_5, T_3: n_0\right\}$ would be valid. Further, the result does not depend on the optimality of the solution recovered, meaning that a solver can be safely interrupted as soon as a feasible solution to $\Pfittzr$ has been found.

The following counterexample illustrates how the result in Theorem~\ref{thm:recursive_feasiblity} might fail when the assumption on non-regressiveness is violated.



\begin{example}[Non-regressiveness]
	\label{example:counterexample_non_regressive_states}
	Consider again the example depicted in Figure~\ref{fig:example_base}. Application of Algorithm~\ref{alg:horizon_computation} in this situation can result in the horizons
	\begin{eqnarray}
	\label{eq:possible_horizons_example}
	f_t = \left[\begin{array}{c}
	f_{T_1}\\
	f_{T_2}\\
	f_{T_3}
	\end{array}\right] = \left[\begin{array}{c}
	6\\
	3\\
	6
	\end{array}\right]
	\end{eqnarray}
	terminating at the nodes indicated with black dashed arrows in Figure~\ref{fig:example_non_regressivity}. Indeed, Figure~\ref{fig:example_train_graph_non_regressivity} presents a feasible movement schedule computed by solving $\Pfittzr$ from this state $\xt$ according to the horizons $\ff_\tim$ in \eqref{eq:possible_horizons_example}.
	
	Suppose that trains depart from their current location at time $\tim$ according to this schedule, and at $\tim + \Delta \tim$ the alternative horizons 
	\begin{eqnarray}
	\ff_{\tim + \Delta \tim} = \left[\begin{array}{c}
	f_{T_1}\\
	f_{T_2}\\
	f_{T_3}
	\end{array}\right] = \left[\begin{array}{c}
	2\\
	6\\
	5
	\end{array}\right]
	\end{eqnarray}
	are selected, as indicated with grayed dotted arrows in Figure~\ref{fig:example_non_regressivity}. The horizons in $\ff_{\tim + \Delta \tim}$ are safe but do not satisfy non-regressiveness: the movement was initiated with an horizon terminating at $n_8$ for $T_1$, while in this subsequent iteration it is regressed to $n_5$. Under these conditions we lose recursive feasibility: train $T_2$ has transit precedence over $T_1$ on shared segments, e.g. $\zedgearg{T_1, T_2, e_{34}} = 0$, but it is stopping at $n_5$, deadlocking $T_1$ and $T_3$ and, hence, ultimately resulting in an infeasible $\Pfitttzzr$.
\end{example}

\begin{figure}
	\centering
	\includegraphics[width=1\linewidth]{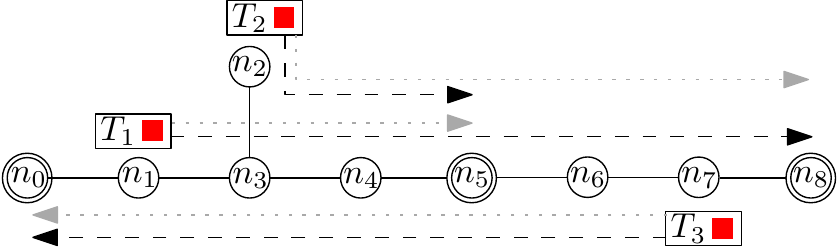}
	\caption{Graph model for the system discussed in Example~\ref{example:counterexample_non_regressive_states}.}
	\label{fig:example_non_regressivity}
\end{figure}
\begin{figure}
	\centering
	\includegraphics[width=0.85\linewidth]{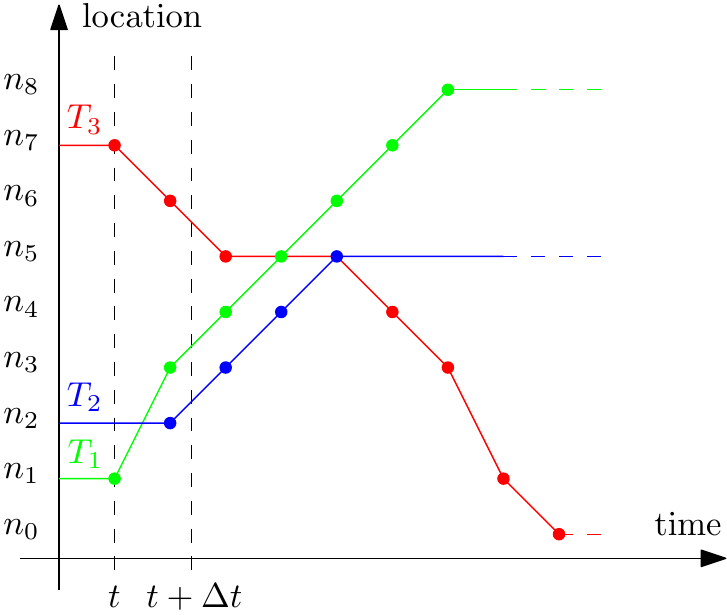}
	\caption{Train graph depicting a possible movement schedule for the system discussed in Example~\ref{example:counterexample_non_regressive_states}.}
	\label{fig:example_train_graph_non_regressivity}
\end{figure}

We finally note that the notion of safe states introduced in Definition~\ref{def:safe_state} does not exclude the existence of more efficient definitions. The one introduced in this paper is sufficient to guarantee recursive feasibility and works well for the freight network discussed in the experimental Section~\ref{sec:results}. Further, networks that do not satisfy the assumptions made in Section~\ref{sec:opt_model_static}, e.g., having more than two parallel lines, may necessitate alternative definitions. The only fundamental requirement is that a safe state must be endowed with a (usually trivial) policy that drives all trains from that safe state into a subsequent safe state, and that this policy can be applied recursively, ensuring that the system can be continuously operated for an ``infinite'' amount of time without deadlocking. In our case this is accomplished by the (inefficient, but valid) policy described in Algorithm~\ref{alg:trivial_safe_to_safe}.


In the following section we discuss important computational ramifications of recursive feasibility.

\section{Computationally Efficient Optimization Approaches}
\label{sec:computational_ramifications_rf}

As mentioned in the introduction, the computation of solutions to $\Pp$ becomes a practical difficulty, in particular when the size of the network and number of trains is large. In this section we take advantage of the tight connection between deadlocking and infeasibility of the underlying optimization models discussed in the previous section to present several approaches to address computational complexity.

\subsection{Warm Starting}
\label{sec:warm_starting}

In warm starting solutions computed at $\tim$ are reused at $\tim + \Delta \tim$ as a system has moved from $\xt$ to $\xzrtt$. We first illustrate how warm starting might fail when the conditions required by Theorem~\ref{thm:recursive_feasiblity} are violated.
\begin{example}[Warm-starting]
	\label{example:counterexample_warmstarting}
	Consider the train graph depicted in Figure~\ref{fig:example_warm_starting} related to the network shown in Figure~\ref{fig:example_base} but with different initial trains positions. Final train destinations are $\left\{T_1: n_0, T_2: n_2, T_3: n_5\right\}$, but in the current optimization model horizons have been truncated as shown on the train graph. They do not satisfy safety as defined in this paper. According to this schedule, $T_1$ transits over $e_{45}$ before $T_2$. At $t+\Delta t$ an optimization model is built with the horizon for $T_2$ extending to $n_2$ and $T_3$ to $n_5$. The only feasible sequence at this stage is for $T_2$ to transit over $e_{45}$ before $T_1$ which is not compatible with the previous solution. Note that the optimization model, in contrast to previous examples, is still feasible at $\tim + \Delta \tim$ as $T_1$ and $T_2$ are still on time to invert transit order in the passing point at $n_5$, but the schedule computed at $t$ is not a valid starting point to seed this optimization.
\end{example}

\begin{figure}
	\centering
	\includegraphics[width=0.60\linewidth]{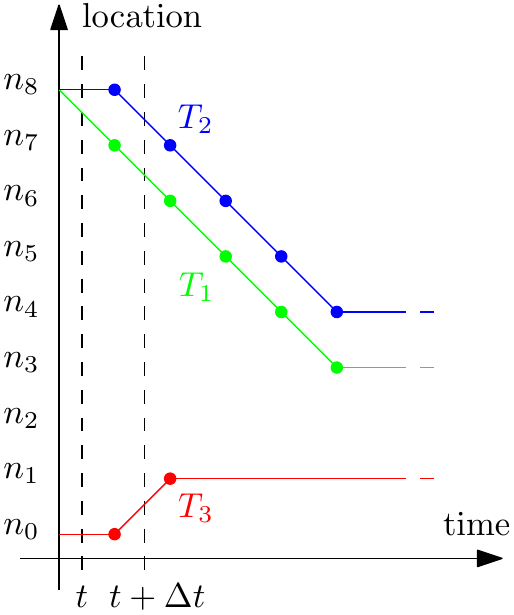}
	\caption{Train graph discussed in Example~\ref{example:counterexample_warmstarting} on warm starting procedures.}
	\label{fig:example_warm_starting}
\end{figure}

The procedures laid out in Theorem~\ref{thm:recursive_feasiblity} for the construction of safe optimization horizons guarantee that warm starting can always be performed. That is, any solution to $\Pfittzr$, when $\Pfittzr$ is feasible and $\ff_\tim$ is computed according to Algorithm~\ref{alg:horizon_computation}, can always be reused at $\tim + \Delta \tim$ as a partial solution to $\Pfitttzzr$. This is shown in the proof of the theorem, when a complete solution to $\Pfitttzzr$ is derived combining the procedure in Algorithm~\ref{alg:trivial_safe_to_safe} with a solution to $\Pfittzr$.

It should be noted that these partial solutions can either be enforced in the following optimization model $\Pfitttzzr$, in which case the size of the model to be solved is reduced, or used only as initialization points for solvers. Both can result in faster computations.

%


A by--product of this result is that Algorithm~\ref{alg:horizon_computation} can be substituted by the more efficient procedure in Algorithm~\ref{alg:horizon_computation_warm} to compute $\ff_\tim$ when the preceding $\ff_{\tim - \Delta \tim}$ is available. In particular, this more effective procedure does not require one to verify the feasibility of $\Pfittzr$ for a candidate $\ff_{\tim}$, as done on line 7 of Algorithm~\ref{alg:horizon_computation}, since the generated $\ff_{\tim}$ is guaranteed to result in a feasible $\Pfittzr$. This is true because, as discussed in Remark~\ref{remark:ftilda_geq_f_always_feasible}, having established that $\Pfitttzzrminus$ is feasible automatically ensures the feasibility of $\Pfittzr$ for $\ff_\tim \geq \ff_{\tim - \Delta \tim}$\footnote{Strictly speaking, Remark~\ref{remark:ftilda_geq_f_always_feasible} ensures the feasibility of $P(t- \Delta t, x_{t- \Delta t}, f_t )$ which, in turn, ensures the required feasibility.}.

%

\begin{algorithm}
	\caption{Warm--Started Computation of $\ff$}
	\label{alg:horizon_computation_warm}
	\begin{algorithmic}[1]  
		\Require \textbf{Required}: $\ff_{\tim - \Delta \tim}$ and $\Niarg{i,\tim - \Delta \tim}$, the node sequence \eqref{eq:nodes_and_edges_in_path} of train $i \in I$ at $\tim - \Delta \tim$.
		\State $\eta_n \gets |L_n|, \ \forall n \in \mathcal{N}$
		\ForAll{$i \in I$}
		\State $\ffi \gets \tia{\Niarg{i,\tim - \Delta \tim}[\ff_{t-\Delta t}]}$ \Comment{$\tia{n}$ is the index of node $n$ in $\Niarg{i,t}$}
		\While{$\eta_{\ni{\ffi}} \leq 1$}{}
		\State $\ffi \gets \ffi+1$
		\EndWhile
		\State $\eta_{\ni{\ffi}} \gets \eta_{\ni{\ffi}} - 1$
		\EndFor
	\end{algorithmic}
\end{algorithm}


\subsection{Anytime approaches}

A direct consequence of the results in Section~\ref{sec:closed_loop_control} and~\ref{sec:warm_starting} is that feasible solutions for arbitrary horizons lengths at subsequent iterations can be obtained without performing optimizations. Namely, once an initial feasible solution to a safe state is found, \emph{cf.}~line 7 of Algorithm~\ref{alg:horizon_computation}, that solution remains valid at $\tim + \Delta \tim$ according to the discussion in Section~\ref{sec:warm_starting}. It can then easily be extended into a solution to any arbitrarily long optimization horizon which satisfies the condition of being safe by application of Algorithm~\ref{alg:trivial_safe_to_safe}. This guarantees that at $t+\Delta t$ a complete feasible solution to $\Pfitttzzr$ is available before any optimization is performed. Solvers can thus always be seeded with an initial feasible solution, and since all results in this paper do not rely on optimality, the solution progress can be interrupted at any time returning valid schedules. 

The quality of the solutions recovered with this approach depends on the quality of the heuristic utilized to move trains from safe state to safe state. The policy in Algorithm~\ref{alg:trivial_safe_to_safe} is evidently suboptimal. It could be improved, for instance, by moving all trains that do not interact which each other at the same time. Generally, we note that designing efficient movement schedules between safe states appears to be simpler than working with generic initial and terminal states. 



\subsection{Time-Wise Problem Decomposition}
\label{subsec:time-wise-decomp}

One important ramification of Theorem~\ref{thm:recursive_feasiblity} is that a feasible solution to $\Pfittzr$, for any arbitrarily long safe optimization horizon $\ff_\tim$, can be obtained by solving a sequence of smaller optimization models, each of which incrementally considers an additional portion of time. More precisely, if $(\overline{y}, \overline{z})$ is any (not necessarily optimal) feasible solution to $\Pfittzr$, then the values of $(\overline{y}, \overline{z})$ are also valid for the optimization problem $\Pfitzrtilde$, where $\tilde{f}_\tim \geq f_\tim$ as long as $\tilde{f}_\tim$ is safe. That is, the values of $(\overline{y}, \overline{z})$ can be forced unto the corresponding variables of $\Pfitzrtilde$ and the latter remains feasible.

This is true because a feasible solution to $\Pfitzrtilde$ can always be constructed by extending a solution to $\Pfittzr$ to any non-regressive safe horizon $\tilde{\ff}_\tim \geq \ff_\tim$ by application of the trivial policy in Algorithm~\ref{alg:trivial_safe_to_safe}, thus guaranteeing feasibility.

Since we are allowed to force all variables $(\overline{y}, \overline{z})$, feasibility guarantees are preserved when we elect to force only a subset of them; in particular, we can only force the binary variables $\overline{z}$. Further, since $(\overline{y}, \overline{z})$ drives the system into a safe state, which might not be an efficient network state, we can exclude variables close to the end of the optimization horizon of the current iteration from the variables forced on subsequent models. Finally, rather than forcing values, this procedure can be used to seed valid initial values for the binary variables in the model.


Note that the chances of time-wise decomposition working on extensive networks with large fleets are exceedingly low when not enforcing safe horizons. As traffic density increases, the likelihood of at least one train terminating at a node that impedes the transit of other trains in subsequent steps also increases (e.g., dwelling on a single slot node).


\subsection{Train-Wise Problem Decomposition}
\label{subsec:train-wise-decomp}


A somewhat unexpected consequence of the results in Section~\ref{sec:recursive_feasibility} is that, under certain provisions, it is possible to solve $\Pp$ by considering only portions of the train fleet at a time.  Namely, let $\Pp_{I_0}$ and $\Pp_{I_1}$ be instances of $\Pp$ only entailing trains in $I_0 \subseteq I$ and $I_1 \subseteq I$ respectively.
We discuss conditions ensuring that the corresponding sub--solutions $(\overline{z}_i)_{i \in I_0}$ and $(\overline{z}_i)_{i \in I_1}$ constitute a valid partial solution\footnote{It is only a partial solution since the model $\Pp_{I_0 \cup I_1}$ entails all conflicts between the trains in $I_0$ and $I_1$ that are not considered in the separate problems} to $\Pp_{I_0 \cup I_1}$. Note that we restrict this analysis to the binary variables $z$. 

Two specific procedures enabled by this result are as follows:
\begin{itemize}
	\item[i)] $I$ is \textit{partitioned} into non-overlapping subsets, i.e., $I_i \cap I_j = \emptyset$ for all partitions $I_i$ and $I_j$. This decomposition allows the construction of a partial feasible solution to $\Pp$ by solving the independent sub-models $\Pp_{I_i}$ in parallel.
	\item[ii)] $I$ is decomposed into incrementally larger subsets, i.e., $I_0 \subseteq I_{1} \subseteq \dots \subseteq I_N$. This decomposition produces solutions to $\Pp$ by considering subsets of trains that are progressively enlarged. If $I_N \equiv I$, this procedure computes the complete set of variables $z$ for $\Pp$.
\end{itemize}
In both cases, at each iteration the size of the problems to be solved are smaller than the full--scale model $\Pp$.

We start with an example illustrating that, as expected, this is generally not possible. We however also show how adjustments to boundary conditions can resolve the underlying issue.

\begin{example}
	\label{example:counterexample_trainwise_decomposition}
	Consider again the example depicted in Figure~\ref{fig:example_base}, and the corresponding graph in Figure~\ref{fig:counterexample_pathwise_decomposition}a, in which initial optimization horizons are shown. Under these circumstances, the only feasible sequence of train movements is to move $T_1$ to node $n_5$ first, then $T_3$ to $n_1$ and finally $T_2$ to $n_5$.
	Suppose, however, that procedure ii) is followed, with the following arbitrary sequence of subsets of $I$: 
	\begin{eqnarray*}
		I_0 = \left\{ T_2, T_3 \right\}, \quad I_1 = \left\{ T_1, T_2, T_3 \right\} \equiv I.
	\end{eqnarray*}
	The optimization model $\Pp_{I_0}$ is not aware of $T_1$, and it can thus elect to give precedence to $T_2$ over $T_3$, i.e., to set $\zedgearg{T_2, T_3, e_{45}} = 1$. Freezing this value for $\zedgearg{T_2, T_3, e_{45}}$ would render the subsequent optimization model $\Pp_{I_1}$ infeasible, where $T_1$ is introduced and $I_1$ is considered.
	Figure~\ref{fig:counterexample_pathwise_decomposition}b illustrates how setting $\zedgearg{T_2, T_3, e_{45}} = 1$ can be resolved in a feasible schedule when horizons are extended to terminate in a safe state. Note that the initial state of the trains is exactly the same as in Figure~\ref{fig:counterexample_pathwise_decomposition}a.
	\begin{figure}
		\centering
		\includegraphics[width=1\linewidth]{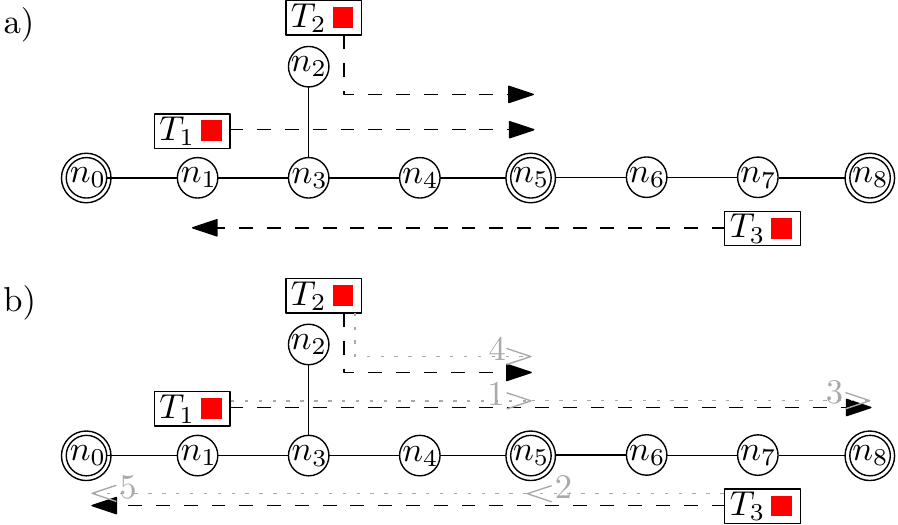}
		\caption{System state discussed in Example~\ref{example:counterexample_trainwise_decomposition}.}
		\label{fig:counterexample_pathwise_decomposition}
	\end{figure}
\end{example}

To see how the desired result might be possible more generally, we first observe that movements of individual trains are almost entirely independent of each other in Algorithm~\ref{alg:trivial_safe_to_safe}. Note that the algorithm assumes that initial and final states are safe. The only coupling between trains in the policy occurs when a train $i$ is moved to its destination node $n$, resulting in all slots in $n$ to be occupied. In this case the policy, as presented, does enforce a specific transit order for scheduling trains by requiring $j$, another train at node $n$ that hasn't been moved yet, to be moved next. Note, however, that it would be possible to rectify this by delaying the departure of all -- or any subset -- of the trains already processed ($I \setminus I^{\mathrm{open}}$) and move $j$ first. The node at which $j$ arrives can itself then be fully occupied, but as before, a train that hasn't moved yet must exist at this node and hence the same procedure can be re-applied. These recursive iterations must terminate because the number of trains that haven't been moved yet is finite.

This demonstrates that the policy can be adapted to return a train transit schedule by processing trains in any sequence and/or independently of each other, provided boundary conditions adhere to safe state requirements. It thus follows that precedences in instances of $\Pp$, in which initial and final states are safe, can be solved by considering conflicts of subsets of trains in any order and, hence, both decomposition schemes mentioned above can be applied. The modified procedure does, however, also require the ability to modify $y$ through iterations, which is why the analysis in this section is exclusively valid for $z$.

Example~\ref{example:counterexample_trainwise_decomposition} violated the assumption on boundary conditions, both for the initial as well the final states. Adjusting the terminal conditions was sufficient to recover feasibility. It is generally possible to make this adjustment whenever optimization horizons can be stretched far enough to reach a safe state, which is always possible under the assumption of infinite capacity at the terminals.

To guarantee that the procedure succeeds in all cases, however, we need to address initial conditions as well.

%

One approach is to run Algorithm~\ref{alg:horizon_computation}, which outputs a minimal safe horizon $\ff_{\tim}$ together with a feasible solution to $\Pfittzr$, and only consider the output $\ff_\tim$. Solve problem $\Pfitzrtilde$ for any safe $\tilde{\ff}_{\tim} \geq \ff_{\tim}$ using either procedure i) or ii) but, at each iteration, only freeze optimization variables indexed from $\ff_{\tim}$ onward; all other variables, which concern the schedule from the trains' current position to the horizon $\ff_{\tim}$, need to be left open as optimization variables. They can, however, be seeded with the values obtained in previous iterations, which, as noted above, will often be a valid initialization point.

This is guaranteed to work because running Algorithm~\ref{alg:horizon_computation} ensures that a feasible schedule exists from the trains' current position to $\ff_\tim$. As long as the existence of at least one solution is guaranteed, the model can be extended from that point using either procedure i) or ii) into a feasible solution to $\Pfitzrtilde$, for any arbitrary $\tilde{f}_\tim \geq \ff_\tim$ that is safe.

An alternative approach is to first construct a feasible schedule from the trains' current state into a safe state. One way to obtain this is to run Algorithm~\ref{alg:horizon_computation} and consider both the horizon $\ff_{\tim}$ as well as the feasible solution to $\Pfittzr$. We can then solve the problem $\Pfitzrtilde$ to any arbitrary $\tilde{f}_\tim \geq \ff_\tim$ by following procedures i) or ii). Note that this approach, however, forces the system to pass through the safe state determined by solving $\Pfittzr$.


It is interesting to remark that this approach works independently of the extent of the network and the complexity of its topology. It is also independent of the train fleet size. The only relevant factors are the initial and terminal conditions, the approach works for any arbitrary complexity degree of traffic patterns between those boundary conditions.

One final note is that the quality of the schedules obtained with this model decomposition depends on the size and sequence of the subsets of $I$ used in the iterations. 


\section{Numerical Experiments and Results}
\label{sec:results}


Two networks were implemented for the numerical experiments presented in this section. One synthetic network comprised of 27 nodes, displayed in Figure~\ref{fig:synthetic_net}, and another network with 69 nodes that models the railway system operating in the Pilbara region for freight transport of mineral ore, depicted in Figure~\ref{fig:pb_net}. The travel times over the edges for the synthetic network with 27 nodes are randomly distributed between 5 and 20 minutes. We cannot disclose travel times used for the network with 69 nodes.

\begin{figure}
	\centering
	\begin{subfigure}[c]{0.45\textwidth}
	\centering
	\includegraphics[width=1\linewidth]{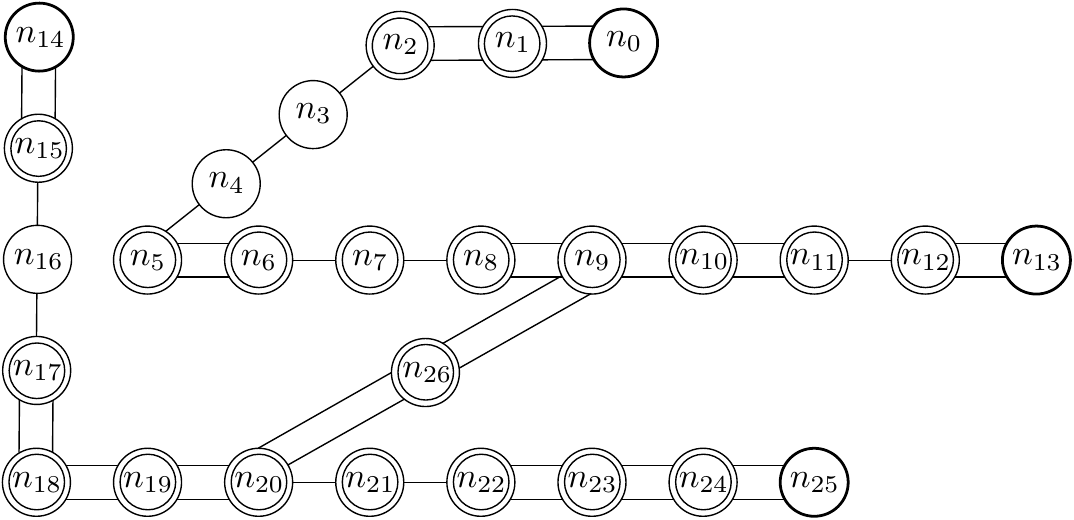}
	\caption{Layout of the synthetic network with 27 nodes. Nodes with single thin circle have one slot, double circles have two slots and single thicker circle have infinite capacity.}
	\vspace{0.25cm}
	\label{fig:synthetic_net}
	\end{subfigure}
	\hspace{0.25cm}
	\begin{subfigure}[c]{0.45\textwidth}
	\centering
	\includegraphics[width=1\linewidth]{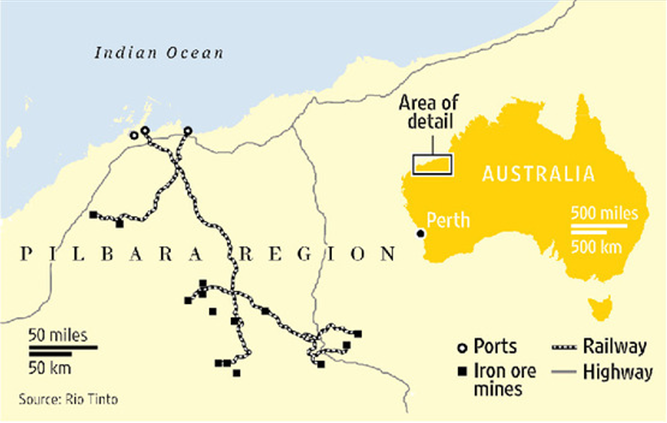}
	\caption{Pilbara railway system used to inspire the model with 69 nodes.}
	\label{fig:pb_net}
	\end{subfigure}
	\caption{Layouts of the networks considered in the experimental section.}
	\label{fig:network_layouts}
\end{figure}

We perform computations varying the number of trains present in the network to assess the sensitivity of computations to traffic levels. For the network with 27 nodes, we consider 10 (moderate traffic), 20 (high traffic) and 30 trains (very high traffic---more trains than nodes). For the 69-node network, we consider 30 and 50 trains. For each network and train number combination, we create 500 random initial positions of trains.
For each random initial condition, we solve $\Pp$ using several algorithms based on the finding presented in the previous section:


\textbf{Time-wise decomposition}. In \textit{time-wise} decompositions, we utilize the results in Section~\ref{subsec:time-wise-decomp}, see Figure~\ref{fig:timewise_decomposition}. The optimization model is split into segments of 30 and 60 minutes, that is, we optimize the model considering a number of edges that is increased at each step in a way that ensures that the total unimpeded travel time is increased by at least 30 or 60 minutes for each train, and extend those further to accommodate for finite, safe horizons. We also consider a variant (``relaxation'') where at each step we relax enforcement of binary variables for the last 15 minutes of the previous solution but, instead, use them only as initialization point.

\textbf{Train-wise decomposition.} In \textit{train-wise} decompositions, we utilize the results from Section~\ref{subsec:train-wise-decomp}, see Figure~\ref{fig:pathwise_decomposition}. The ``incremental'' version refers to variant i), while ``partitions'' corresponds to variant ii). We run experiments with varying sizes of the train subsets considered at each step. To make comparisons fair, since the ``partitions'' strategy only recovers a partial solution to $\Pp$, we perform a last step in which that partial solution is enforced into the full model $\Pp$ to retrieve a complete solution. The trains selected to be within the next subset at each iteration are chosen randomly.

\textbf{Monolithic.} In the \textit{monolithic} version, $\Pp$ is solved as a single optimization model until the incumbent solution has a guaranteed optimality gap of less than 0.1\% or 120 seconds have elapsed, whichever occurs first.

The results of these experiments are presented in Figures~\ref{fig:results_synthetic} and~\ref{fig:results_pilbara}, for 27 and 69 nodes networks respectively. All optimizations are performed using Gurobi~\cite{gurobi} on a PC with 16GB of RAM, an Intel i7-6700K CPU clocking at 4.00GHz running on Linux Ubuntu 16.04.4 LTS.  Computation times presented in the results only consider the time spent performing optimizations (total Gurobi runtime), as the time spent setting up the various algorithms is highly dependent on implementation details and does not entail exponential complexities (as opposed to solving mixed-integer models). Computation times are floored at 0.01 sec to enable logarithmic plotting.
 
For the synthetic network and moderate traffic case (10 trains), all methods quickly ($\leq 0.1$ sec) solve the model to optimality in the vast majority of cases. For cases with higher traffic, it is possible to distinguish a trade-off between computation time and solution quality. Approaches with higher compute times are generally producing higher quality solutions and vice-versa. The monolithic variant tends to produce solutions with lowest optimality gaps, while median compute times for incremental variants of train-wise decompositions are two orders of magnitude faster, while retaining (in median) optimality gaps of 5\% or less. 

Computational constraints for the network with 69 nodes make this a more challenging set of instances, especially the experiments involving 50 trains. For this case, the monolithic approach caps at the maximum allowed compute time of 120 seconds in the majority of instances, and presents several outliers with high optimality gaps. Incremental train--wise decompositions with 1 and 5 trains per subset significantly outperform this approach in terms of worst-case optimality gap while being more than two orders of magnitude faster in terms of median computation times.

The results indicate that problems' hardness is very strongly related to the traffic density in the network: significant increases in compute times can be observed for all algorithms on both networks as the number of trains is increased.  In particular, median compute times for the monolithic variant grow in excess of an order of magnitude at each higher level of traffic density for both networks. Note also that compute times for the synthetic network with 30 trains are approximately an order of magnitude slower than for the 69 nodes network with the same number of trains. This is to be expected as the number of conflicts (and hence binary variables) grows with increased interactions between the trains.

Comparing time-wise decompositions, we note that performing relaxations drastically improves solutions' quality, while retaining median compute times that are approximately one order of magnitude faster than the monolithic approach for the cases with the highest traffic. Improved quality is likely due to the fact that solutions are not forced through a safe state at the end of each solution step.

Train-wise decomposition with partitions tend to require more computations than the incremental variant and this is mainly due to the last step, in which a full solution is computed from a partial one. All prior steps, involving separate and independent partitions, compute very quickly.

\begin{figure*}
	\centering
	\begin{subfigure}[c]{0.32\textwidth}
		\centering
		\includegraphics[width=1\linewidth]{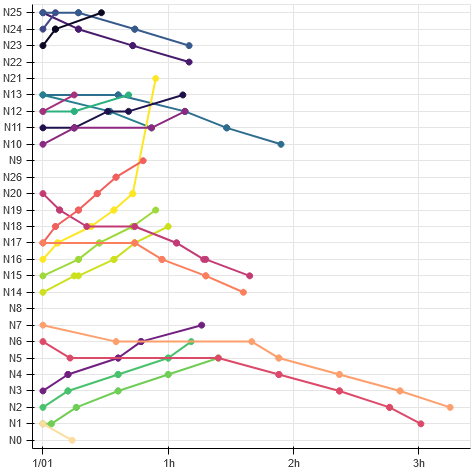}
	\end{subfigure}
	\hspace{0.25cm}
	\begin{subfigure}[c]{0.32\textwidth}
		\centering
		\includegraphics[width=1\linewidth]{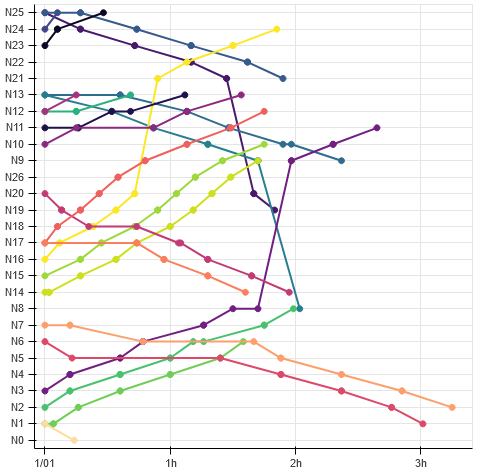}
	\end{subfigure}
	\begin{subfigure}[c]{0.32\textwidth}
		\centering
		\includegraphics[width=1\linewidth]{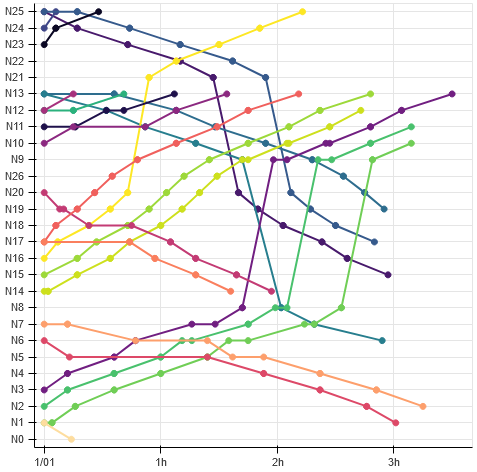}
	\end{subfigure}
	\caption{Three iterations of the time-wise decomposition solution approach. At each step, the movements schedule is extended by at least 60 minutes. Note in the first iteration how the horizon for the trains departing from $N6$ and $N7$ is extended further than the rest: after 60 minutes, they would occupy $N5$ and $N6$, both of which have two slots but are already terminal for the trains departing from $N1$ and $N2$. Nodes $N3$ and $N4$ have only one slot so they cannot function as terminal nodes. Horizons are consequently extended up to $N1$ and $N2$, both of which have two slots and are not terminal for other trains.}
	\label{fig:timewise_decomposition}
\end{figure*}

\begin{figure*}
	\centering
	\begin{subfigure}[c]{0.32\textwidth}
		\centering
		\includegraphics[width=1\linewidth]{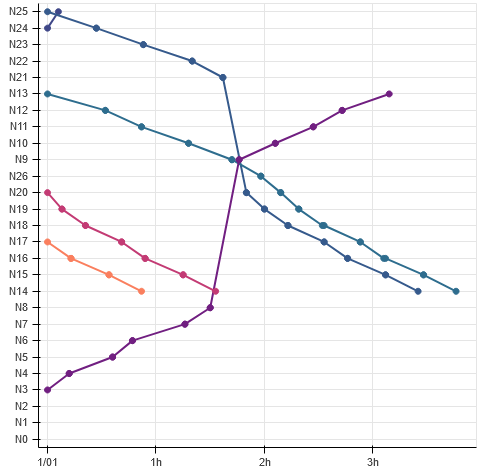}
	\end{subfigure}
	\hspace{0.25cm}
	\begin{subfigure}[c]{0.32\textwidth}
		\centering
		\includegraphics[width=1\linewidth]{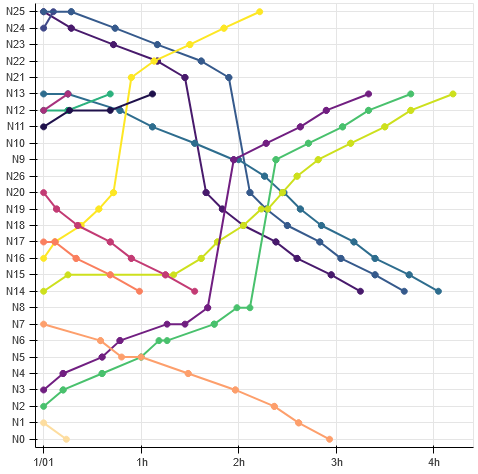}
	\end{subfigure}
	\begin{subfigure}[c]{0.32\textwidth}
		\centering
		\includegraphics[width=1\linewidth]{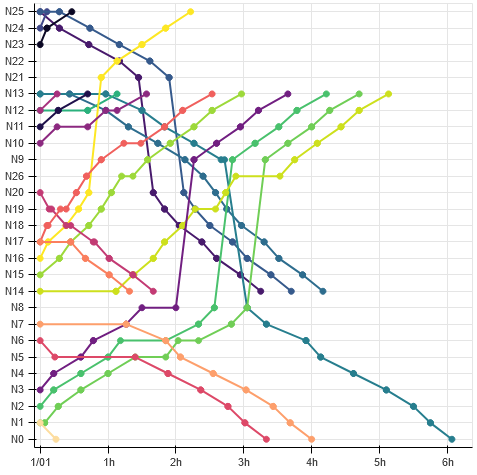}
	\end{subfigure}
	\caption{Three iterations of the train-wise decomposition solution approach. At each iteration, an additional subset of trains is added to the model while previously established precedences are frozen.}
	\label{fig:pathwise_decomposition}
\end{figure*}

\begin{figure*}[t]
	\begin{subfigure}[b]{1.0\textwidth}
		\centering
		\includegraphics[width=1\textwidth]{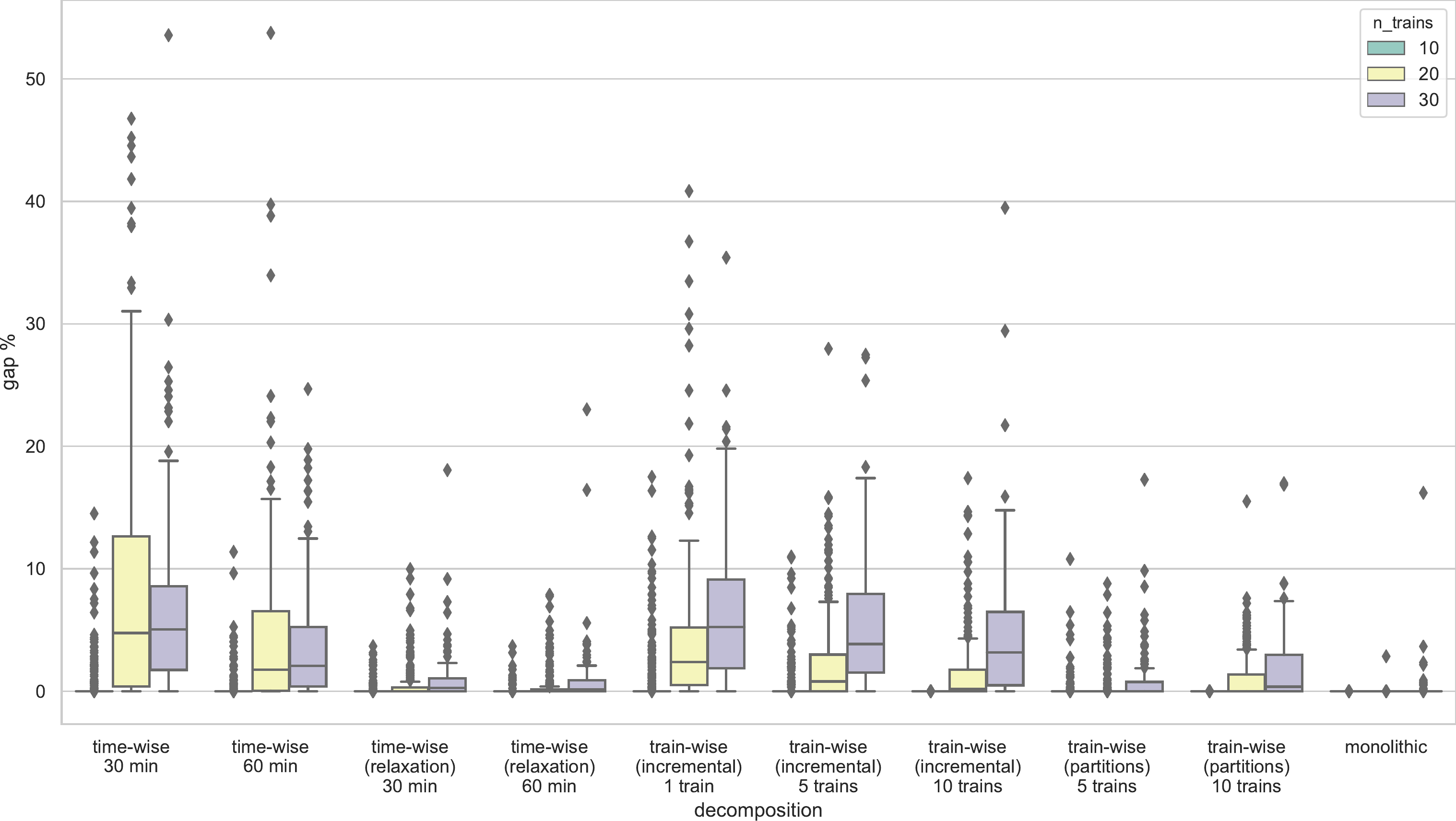}		
		\vspace{0.25cm}
	\end{subfigure}
	\begin{subfigure}[b]{1.0\textwidth}
		\centering
		\includegraphics[width=1\textwidth]{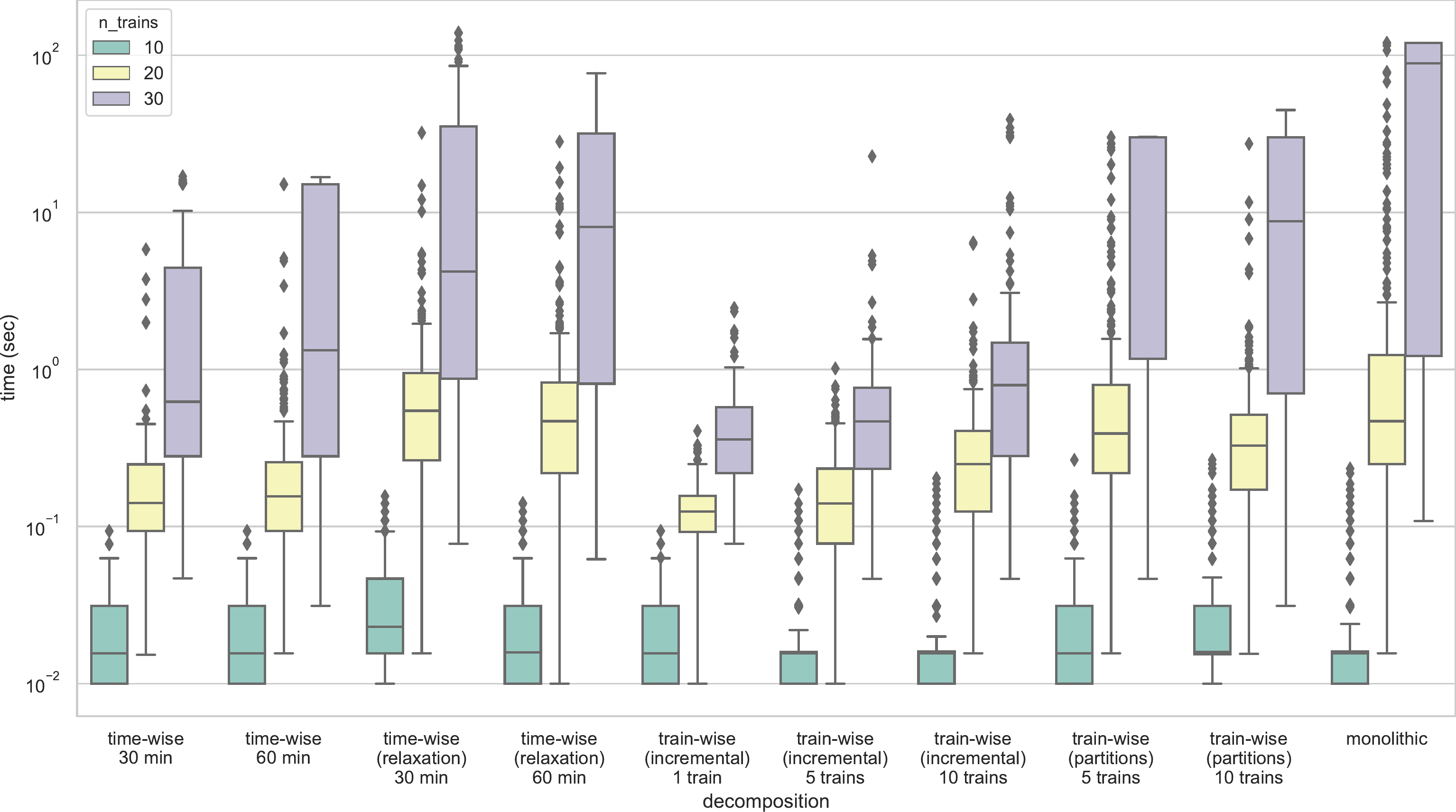}		
		\vspace{0.25cm}
	\end{subfigure}
	\caption{Results for the synthetic network with 27 nodes.}
	\label{fig:results_synthetic}
\end{figure*}

\begin{figure*}[t]
	\begin{subfigure}[b]{1.0\textwidth}
		\centering
		\includegraphics[width=1\textwidth]{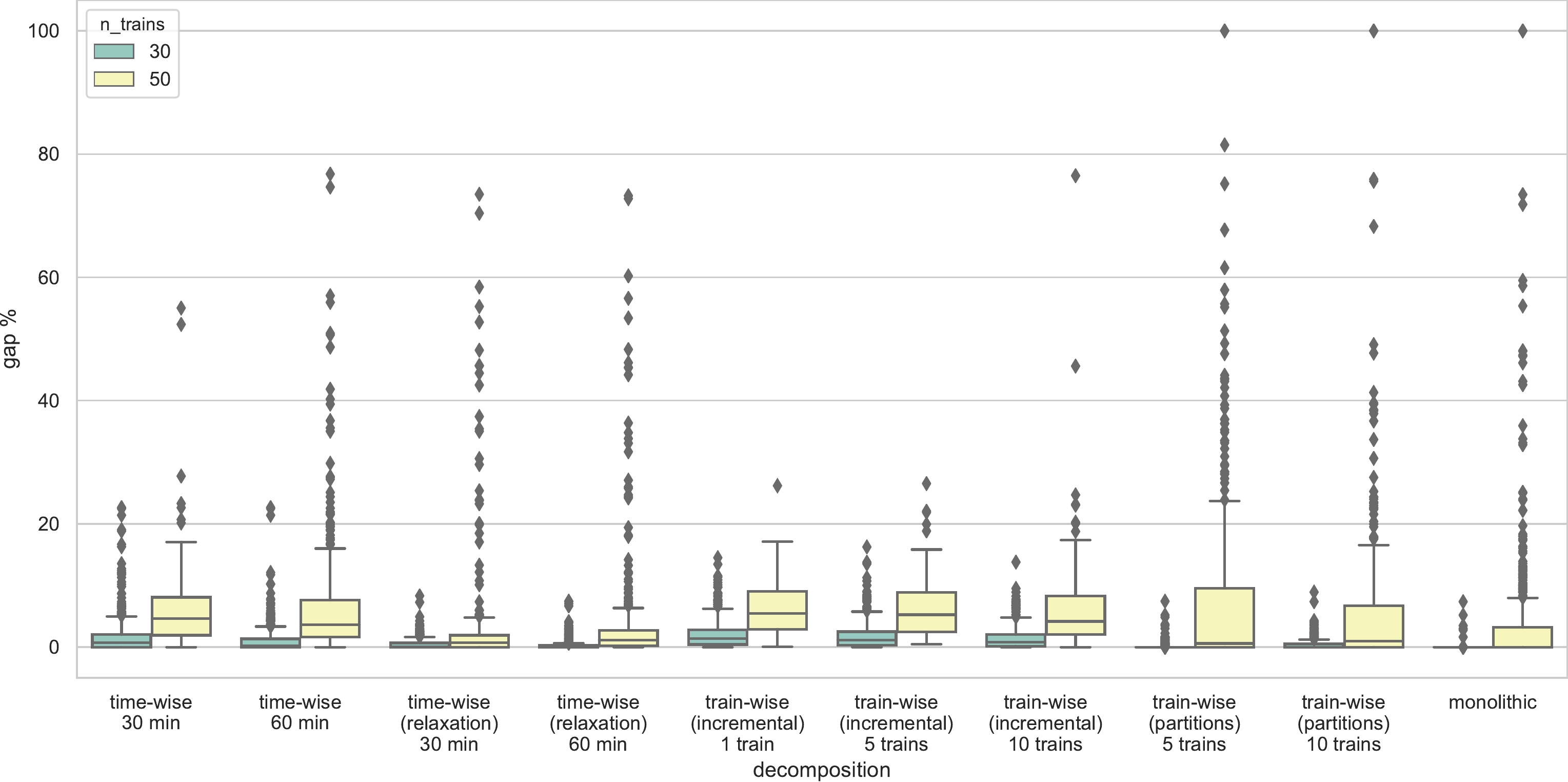}		
	\end{subfigure}
	\begin{subfigure}[b]{1.0\textwidth}
		\centering
		\includegraphics[width=1\textwidth]{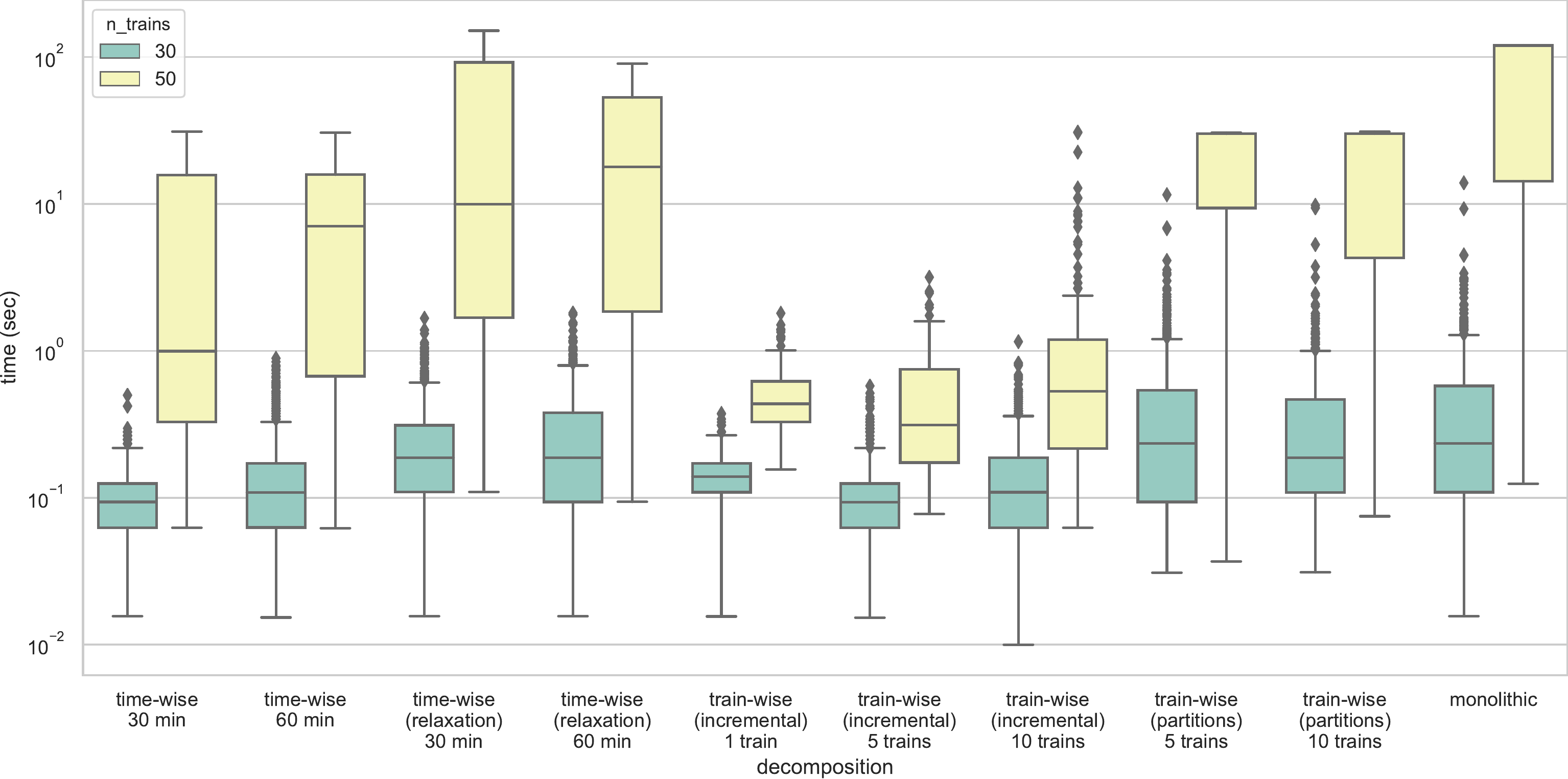}		
	\end{subfigure}
	\caption{Results for network with 69 nodes.}
	\label{fig:results_pilbara}
\end{figure*}

\section{Conclusions}

This paper presented a mathematical optimization model for the optimization of traffic in railway systems focusing on the operation of such optimization model in a receding horizon framework. Theoretical guarantees are provided that determine the required optimization horizons necessary to operate the system in a deadlock-free manner. Several important strategies for computational complexity reduction that follow from these ideas are presented. The effectiveness of these procedures is demonstrated through extensive computational experimentation that shows that in larger networks with high traffic density median computation times can be reduced by more than two orders of magnitude while retaining lower worst--case optimality gaps than solving the same models as monolithic MILPs.

We note that the analysis in Section~\ref{sec:computational_ramifications_rf} does not depend on the algorithms utilized to solve the underlying optimization models. 
Hence, the ideas presented can be useful in contexts in which train schedules are computed using different solution methods. In this context, it is worth noting that the fundamental idea of a safe state might be relevant also for other systems that do not satisfy the assumptions made in Section~\ref{sec:opt_model_static}. The specific definition of safe states introduced in this paper may not be suitable for these systems, but alternatives may be devised. This is an open research topic. Further, the interplay between the results presented in this paper and cases where the railway system is driven according to an underlying timetable and where online control and optimization is rather used to address disturbances is also an open area of investigation.

Future avenues of research include investigations of more efficient algorithms to move trains between safe states and extensions to avoid the assumption of infinite capacity at the terminals. Procedures to address system disturbances (e.g., delayed travel times) as well as imperfect information are also of utmost practical importance and they could be investigated using approaches from control systems theory, e.g. robust model predictive control~\cite{morari_robust_mpc}. Aspects related to the selection strategies of train subsets in train-wise decompositions are also an interesting line of further research. Finally, questions on whether it is possible to provide performance guarantees for the solutions produced by any of the optimization procedures presented in the paper are also to be further analyzed.

\textbf{Acknowledgments.} This work has been supported by the Australian Centre for Field Robotics and the Rio Tinto Centre for Mine Automation. The Authors would like to especially thank Shaun Robertson, Rio Tinto Growth \& Innovation, for his guidance throughout the project on practical aspects of railway systems.

\bibliographystyle{amsplain}
\bibliography{mybibliography}

%
%
%

\end{document}